\documentclass[11pt]{amsproc}

\usepackage{url}

\usepackage[backrefs]{amsrefs}

\usepackage{enumerate}
\usepackage{csquotes}

\newcommand{\newnotion}[1]{\emph{#1}}

\usepackage{xargs}
\usepackage{amsmath,amssymb,amsfonts,amsthm}
\usepackage{mathtools}

\newcommandx{\gennorsubgrp}[2][2=\empty]{\langle \!\langle#1\ifx#2\empty\else\mid#2\fi \rangle \! \rangle}
\newcommandx{\genidl}[2][2=\empty]{(#1\ifx#2\empty\else\mid#2\fi)}
\DeclareMathOperator{\Aut}{Aut}
\newcommand{\completion}[1]{\widehat{#1}}
\newcommand{\centralizersubgrp}{\mathbf{C}}
\newcommand{\supgrpeq}{\geq}

\DeclareMathOperator{\diam}{diam}
\DeclareMathOperator{\rk}{rk}
\newcommand{\freegrp}{\mathbf{F}}
\newcommand{\trivgrp}{\mathbf{1}}
\newcommand{\lcssubgrp}{\gamma}
\newcommandx{\gensubgrp}[2][2=\empty]{\langle#1\ifx#2\empty\else\mid#2\fi\rangle}
\newcommandx{\lrgensubgrp}[2][2=\empty]{\left\langle#1\ifx#2\empty\else\mid#2\fi\right\rangle}
\newcommand{\subgrpeq}{\leq}
\newcommand{\norsubgrpeq}{\trianglelefteq}
\newcommand{\opnorsubgrpeq}{\trianglelefteq_{\mathrm{o}}}
\newcommand{\con}{\equiv}
\newcommandx{\seq}[2][2=\empty]{{(#1\ifx#2\empty\else:#2\fi)}}
\newcommandx{\lrseq}[2][2=\empty]{{\left(#1\ifx#2\empty\else:#2\fi\right)}}
\newcommandx{\set}[2][2=\empty]{\{#1\ifx#2\empty\else\mid#2\fi\}}
\newcommandx{\lrset}[2][2=\empty]{\left\{#1\ifx#2\empty\else\mid#2\fi\right\}}

\newcommand{\compose}{\circ}
\newcommand{\commutator}[1]{[#1]}

\newcommandx{\grp}[2][2=\empty]{{\langle#1\ifx#2\empty\else\mid#2\fi\rangle}}
\newcommandx{\lrgrp}[2][2=\empty]{{\left\langle#1\ifx#2\empty\else\mid#2\fi\right\rangle}}
\newcommand{\nathom}[1]{\overline{#1}}
\newcommand{\closure}[1]{\overline{#1}}
\renewcommand{\implies}{\Rightarrow} 
\newcommand{\defeq}{\coloneqq}
\newcommand{\eqq}{\con}
\newcommand{\card}[1]{\lvert#1\rvert}
\newcommand{\setjoin}{\cup}
\newcommand{\setmeet}{\cap}
\newcommand{\bigsetmeet}{\bigcap}
\newcommand{\modulus}[1]{\left\lvert #1\right\rvert}

\newcommand{\catFin}{\mathbf{Fin}}
\newcommand{\catAlt}{\mathbf{Alt}}
\newcommand{\catAb}{\mathbf{Ab}}
\newcommand{\catNil}{\mathbf{Nil}}
\newcommand{\catSol}{\mathbf{Sol}}

\newcommand{\catPSL}{\mathbf{PSL}}

\newcommand{\nats}{\mathbb{N}}
\newcommand{\reals}{\mathbb{R}}
\newcommand{\complex}{\mathbb{C}}
\newcommand{\ints}{\mathbb{Z}}

\DeclareMathOperator{\PSL}{PSL}
\DeclareMathOperator{\PGL}{PGL}
\DeclareMathOperator{\SL}{SL}
\DeclareMathOperator{\SO}{SO}

\DeclareMathOperator{\GF}{GF}
\newcommand{\into}{\hookrightarrow}

\newcommand{\calB}{\mathcal{B}}
\newcommand{\calC}{\mathcal{C}}
\newcommand{\calN}{\mathcal{N}}
\newcommand{\calU}{\mathcal{U}}

\newcommand{\rmP}{\mathrm{P}}
\newcommand{\rmc}{\mathrm{c}}
\newcommand{\rmd}{\mathrm{d}}

\theoremstyle{plain}
\newtheorem{lemma}{Lemma}
\newtheorem{theorem}{Theorem}
\newtheorem{definition}{Definition}
\newtheorem{corollary}{Corollary}

\theoremstyle{definition}
\newtheorem{remark}{Remark}
\newtheorem{example}{Example}

\newtheorem{claim}{Claim}

\title{Some remarks on finitarily approximable groups}

\author{Nikolay Nikolov}
\address{N.~Nikolov, University of Oxford, OX2 6GG Oxford, UK}
\email{nikolay.nikolov@maths.ox.ac.uk}

\author{Jakob Schneider}
\address{J.~Schneider, TU Dresden, 01062 Dresden, Germany}
\email{jakob.schneider@tu-dresden.de}

\author{Andreas Thom}
\address{A.~Thom, TU Dresden, 01062 Dresden, Germany}
\email{andreas.thom@tu-dresden.de}

\begin{document}

\begin{abstract}
The concept of a $\calC$-approximable group, 
for a class of finite groups $\calC$, is a common generalization of the concepts of a sofic, weakly sofic, and linear sofic group.

Glebsky raised the question whether all groups are approximable by finite solvable groups with arbitrary invariant length function. We answer this question by showing that any non-trivial finitely generated perfect group does not have this property, generalizing a counterexample of Howie. On a related note, we prove that any non-trivial group which can be approximated by finite groups has a non-trivial quotient that can be approximated by finite special linear groups.

Moreover, we discuss the question which connected Lie groups can be embedded into a metric ultraproduct of finite groups with invariant length function. We prove that these are precisely the abelian ones, providing a negative answer to a question of Doucha. Referring to a problem of Zilber, we show that a the identity component of a Lie group, whose topology is generated by an invariant length function and which is an abstract quotient of a product of finite groups, has to be abelian. Both of these last two facts give an alternative proof of a result of Turing. Finally, we solve a conjecture of Pillay by proving that the identity component of a compactification of a pseudofinite group must be abelian as well.

All results of this article are applications of theorems on generators and commutators in finite groups by the first author and Segal. In Section \ref{sec:Fin_approx_grps_hom_PSL} we also use results of Liebeck and Shalev on bounded generation in finite simple groups.
\end{abstract}
	
\maketitle


\vspace{-0.6cm}
\section{Introduction}

Eversince the work of Gromov on Gottschalk's Surjunctivity Conjecture \cite{gromov1999endomorphisms}, the class of sofic groups has attracted much interest in various areas of mathematics. Major applications of this notion arose in the work Elek and Szab\'o  on Kaplansky's Direct Finiteness Conjecture \cite{elekszabo2004sofic}, L\"uck's Determinant Conjecture \cite{elekszabo2005hyperlinearity}, and more recently in joint work of the third author with Klyachko on generalizations of the Kervaire-Laudenbach Conjecture and Howie's Conjecture \cite{klyachkothom2017new}.

Despite considerable effort, no non-sofic group has been found so far. In view of this situation, attempts have been made to provide variations of the problem that might be more approachable. In the terminology of Holt and Rees, sofic groups are precisely those groups which can be approximated by finite symmetric groups with normalized Hamming length (in the sense of Definition 1.6 from \cite{thom2012metric}). It is natural to vary the class of finite groups and also the metrics that are allowed. Note that our terminology differs from the one used in \cite{ouldhoucinepoint2013alternatives}, where similar concepts were studied.

The strongest form of approximation is satisfied by LEF (resp.~LEA) groups. In this case, it is well known that a finitely presented group is not approximable by finite (resp.~amenable) groups with discrete length function, i.e.~it is not LEF (resp.~LEA), if and only if it fails to be residually finite (resp.~residually amenable). 
Examples of sofic groups which fail to be LEA (and thus also fail to be LEF) are given in \cite{cornulier2011sofic} and \cite{karnikolov2014non} (see also \cite{thom2008examples}), answering a question of Gromov \cite{gromov1999endomorphisms}.

In \cite{thom2012metric} the third author proved that the so-called Higman group cannot be approximated by finite groups with commutator-contractive invariant length function.
In \cite{howie1984the} Howie presented a group which, by a result of Glebsky \cite{glebskyrivera2008sofic}, turned out not to be approximable by finite nilpotent groups with arbitrary invariant length function.

The present article provides four more results of this type (see Sections \ref{sec:Sol_approx_grps} and \ref{sec:approx_Lie_grps}). However, in our setting we restrict only the classes of finite groups and do not impose restrictions on the length functions of the approximating groups other than being invariant (see Definitions \ref{def:C_approx_abs_grp} and \ref{def:C_approx_top_grp}).

Recently, Glebsky \cite{glebsky2016approximations} asked whether all groups can be approximated by finite solvable groups (in the sense of Definition \ref{def:C_approx_abs_grp}). In Section \ref{sec:Sol_approx_grps} we answer this question by establishing that each non-trivial finitely generated perfect group is a counterexample (see Theorem \ref{thm:fin_gen_perf_grp_not_C_approx}). The key to this result is a theorem of Segal \cite{segal2000closed} on generators and commutators in finite solvable groups.

In Section \ref{sec:Fin_approx_grps_hom_PSL}, using results of the first author from \cite{nikolovsegal2012generators} and of Liebeck and Shalev from \cite{liebeckshalev2001diameters}, we prove that any non-trivial group which is approximable by finite groups has a non-trivial homomorphism into a metric ultraproduct of finite simple groups of type $\PSL_n(q)$ with conjugacy length function (see Theorem \ref{thm:smpl_Fin_approx_grp_ultraprod_smpl_grps}). 

In Section \ref{sec:approx_Lie_grps} we discuss the approximability of Lie groups by finite groups. It is easy to see that $\reals$ as a topological group is not approximable by symmetric groups, i.e.~ it is not continuously embeddable into a metric ultraproduct of symmetric groups with arbitrary invariant length function (see Remark \ref{rmk:reals_not_Sym_approx}).
Using a much deeper analysis, we show that a connected Lie group is appoximable by finite groups (in the sense of Definition \ref{def:C_approx_top_grp}) precisely when it is abelian (see Theorem \ref{thm:ctd_Fin_approx_Lie_grps_ab}).
In Question 2.11 of \cite{doucha2016metric} Doucha asked for groups which can be equipped with an invariant length function such that they do not embed into a metric ultraproduct of finite groups with invariant length function. Our result implies that any compact, connected, and non-abelian Lie group is an example of such a group. Thus the simplest example of a topological group which is not weakly sofic, i.e.~not continuously embeddable into a metric ultraproduct of finite groups with invariant length function, is $\SO_3(\reals)$. However, we remark that every linear Lie group is an abstract subgroup of the algebraic ultraproduct of finite groups indexed over $\nats$ (see Remark \ref{rmk:Lie_grp_not_Fin_approx_top_matters}).

Furthermore, in the same section we answer the question of Zilber \cite{zilber2014perfect} if there exist a compact simple Lie group which is not a quotient of an algebraic ultraproduct of finite groups. Indeed, we show that a Lie group which can be equipped with an invariant length function generating its topology and which is an abstract quotient of a product of finite groups has abelian identity component (see Theorem \ref{thm:Lie_grp_quot_prod_fin_grps}). Hence any compact simple Lie group fails to be approximable by finite groups in the sense of Zilber. 

A slight variation of Theorem \ref{thm:Lie_grp_quot_prod_fin_grps} also answers Question 1.2 Pillay \cite{pillay2015remarks}. Moreover, we point out that Theorem \ref{thm:ctd_Fin_approx_Lie_grps_ab} and Theorem \ref{thm:Lie_grp_quot_prod_fin_grps} provide an alternative proof of the main result of Turing \cite{turing1938finite}. 

Finally, using the same approach as for the previous two results, we solve the conjecture of Pillay \cite{pillay2015remarks} that the Bohr compactification of any pseudofinite group is abelian (see Theorem \ref{thm:cpt_pseudo_fin_grp_ab}).

All results of Section \ref{sec:approx_Lie_grps} follow from a theorem on generators and commutators in finite groups of the first author and Segal \cite{nikolovsegal2012generators}. 

\section{Preliminaries}

In this section we recall some basic concepts. We introduce the notion of $\calC$-approximable abstract and topological groups, and present examples.

\subsection{Metric groups and length functions}

By a metric group we mean a group equipped with a metric. However, we allow a metric to attain the value infinity (this is needed for Definition \ref{def:met_ultraprod_fin_grps} to make sense).
For a left-invariant metric $d_G:G\times G\to[0,\infty]$ on the group $G$ we define the corresponding \newnotion{length function} (or \newnotion{norm}) $\ell_G:G\to[0,\infty]$ by $\ell_G(g)\defeq d_G(1_G,g)$. It inherits the following three properties from $d_G$:
\begin{enumerate}[(i)]
	\item $\ell_G(g)=0$ if and only if $g=1_G$;
	\item $\ell_G(g)=\ell_G(g^{-1})$ for $g\in G$;
	\item $\ell_G(gh)\leq\ell_G(g)+\ell_G(h)$ for $g,h\in G$.
\end{enumerate}
Conversely, to any such function $\ell_G$ we associate a left-invariant metric $d_G:G\times G\to [0,\infty]$ defined by $d_G(g,h)\defeq\ell_G(g^{-1}h)$. This gives a one-to-one correspondence between length functions and left-invariant metrics. We indicate that a length function and a left-invariant metric correspond to each other by equipping them with the same decoration.

The length function $\ell_G$ is called \newnotion{invariant} if it is constant on conjugacy classes. This happens precisely when $d_G$ is bi-invariant. In this article all length functions will be invariant (and all metrics are bi-invariant).

Let us introduce the following types of length functions on a finite group $G$, which we shall use in this article:
The \newnotion{discrete length function} $\ell_G^{\rmd}$ is the simplest one. It is defined by $\ell_G^{\rmd}(g)\defeq 1$ for $g\in G\setminus\set{1_G}$ and corresponds to the discrete metric $d_G^{\rmd}$ on $G$.
The \newnotion{conjugacy (pseudo) length function} $\ell_G^{\rmc}$ is defined by
\begin{equation} \label{conjmetric}
\ell_G^{\rmc}(g)\defeq \frac{\log\card{g^G}}{\log\card{G}}
\end{equation}
for $g\in G$, where $g^G$ is the conjugacy class of $g$ in $G$. It is only a proper length function if $G$ has trivial center.
The \newnotion{projective rank length function} $\ell_G^{\rm pr}$ is defined if $G\subgrpeq\PGL_n(q)$ for some $n\in\nats$ and $q$ a prime power; 
$$
\ell_G^{\rm pr}(g)\defeq\frac{1}{n}\min\set{\rk(1-\hat{g})}[\hat{g}\text{ some lift of } g].
$$
Finally, the \newnotion{Cayley length function} $\ell_G^{{\rm Cay},S}$ with respect to some subset $S\subseteq G$ is defined by 
$$
\ell_G^{{\rm Cay},S}(g)\defeq\min\set{n\in\nats}[g=s_1\cdots s_n\text{ for }s_i\in S\setjoin S^{-1}]\setjoin\set{\infty}.
$$

We call a family of length functions $\seq{\ell_i}_{i\in I}$ on a sequence of finite groups $\seq{H_i}_{i\in I}$ Lipschitz continuous with respect to a second family $\seq{\ell_i'}_{i\in I}$ on the same groups if there is $L>0$ such that $\ell_i\leq L\ell_i'$ ($i\in I$). 

For example, since $\ell^{\rm c}_G,\ell^{\rm pr}_G\leq\ell^{\rm d}_G$ for any finite group $G$, the conjugacy length function and the projective rank length function are Lipschitz continuous with respect to the discrete length function (with $L=1$).

If $\seq{\ell_i'}_{i\in I}$ is also Lipschitz continuous with respect to $\seq{\ell_i}_{i\in I}$, we call these families Lipschitz equivalent. 

For example, $\ell^{\rm c}$ and $\ell^{\rm pr}$ are Lipschitz equivalent on the class of non-abelian finite simple groups, which follows from \cite{liebeckshalev2001diameters} (see the argument at the end of Section \ref{sec:Fin_approx_grps_hom_PSL}).

\subsection{On $\calC$-approximable abstract groups}

Now we can define metric approximation of an abstract group by a class of finite groups. Throughout the article let $\calC$ be such a class.

\begin{definition}\label{def:C_approx_abs_grp}
	An abstract group $G$ is called \newnotion{$\calC$-approximable} if there is a function $\delta_{\bullet}:G\setminus\set{1_G}\to(0,\infty]$ such that for any finite subset $S\subseteq G$ and $\varepsilon>0$ there exist a group $H\in\calC$, an invariant length function $\ell_H$ on $H$, and a map $\varphi:S\to H$, such that
	\begin{enumerate}[(i)]
		\item if $1_G\in S$, then $\varphi(1_G)=1_H$;
		\item if $g,h,gh\in S$, then $d_H(\varphi(g)\varphi(h),\varphi(gh))<\varepsilon$;
		\item for $g\in S\setminus\set{1_G}$ we have $\ell_H(\varphi(g))\geq\delta_g$.
	\end{enumerate}
\end{definition}

Note that the above definition differs slightly from Definition 1.6 from the one in \cite{thom2012metric} as we impose no restrictions on the invariant length functions. However, it is equivalent to Definition 6 from \cite{glebsky2016approximations}.

Indeed, we may even require that $\ell_H\leq 1$ and $\delta_\bullet\eqq 1$ in the above definition without changing its essence. Namely, choosing $\varepsilon>0$ small enough, setting $c\defeq\min_{g\in S}{\delta_g}$, $\ell_H'\defeq\min\set{\ell_H/c,1}$, and defining $\delta':G\setminus\set{1_G}\to(0,\infty]$ by $\delta_\bullet'\defeq 1$, we can replace $\delta$ by $\delta'$ and $\ell_H$ by $\ell_H'$. So, in the sense of \cite[p.~3]{holtrees2016some}, if we do not impose restrictions on the length functions on the groups from $\calC$, the terms \enquote*{\newnotion{$\calC$-approximation property}}, \enquote*{\newnotion{discrete $\calC$-approximation property}}, and \enquote*{\newnotion{strong $\calC$-approximation property}} coincide.

Moreover, similar to soficity, being $\calC$-approximable is a local property. This is expressed in the following remark.

\begin{remark}\label{rmk:C_approx_iff_fin_gen_subgrp_C_approx}
	An abstract group is $\calC$-approximable if and only if every finitely generated subgroup has this property.
\end{remark}

Let us now present some examples of $\calC$-approximable abstract groups. Subsequently, denote by $\catAlt$ (resp.~$\catFin$) the class of finite alternating groups (resp.~the class of all finite groups). Indeed, $\calC$-approximable abstract groups (in the above sense) can be seen as a generalization of sofic (resp.~weakly sofic) groups as it is shown in Section 2 of \cite{glebsky2016approximations}:

\begin{example}
	A group is sofic (resp.~weakly sofic) if and only if it is $\catAlt$-approximable (resp.~$\catFin$-approximable) as an abstract group.
\end{example}

Groups approximable by certain classes of finite simple groups of Lie type have been studied in \cite{arzhantsevapaunescu2017linear} and \cite{thomwilson2014metric, thomwilson2016some}.

Every $\calC$-group $G$ is certainly $\calC$-approximable, since we can take $H\defeq G$, $\varphi$ to be the restriction of the identity on $G$ to $S$, and $\ell_H\defeq\ell_H^{\rmd}$ to be the discrete length function on $H$ in Definition \ref{def:C_approx_abs_grp}. Hence Remark \ref{rmk:C_approx_iff_fin_gen_subgrp_C_approx} implies:

\begin{example}\label{exl:loc_C_impl_C_approx}
	Every locally $\calC$-group is $\calC$-approximable as an abstract group.
\end{example}

%
%

\subsection{Metric ultraproducts of groups}

Since there is another common equivalent characterization of $\calC$-approximable groups via \newnotion{metric ultraproducts} of $\calC$-groups with invariant length function, we recall this concept next (for more details on the algebraic and geometric structure of such ultraproducts see also \cite{stolzthom2014lattice} and \cite{thomwilson2014metric, thomwilson2016some}).

\begin{definition}\label{def:met_ultraprod_fin_grps}
	Let $\seq{H_i,\ell_i}_{i\in I}$ be a sequence of finite groups with invariant length function and be $\calU$ an ultrafilter on the index set $I$. The metric ultraproduct $(\nathom{H},\nathom{\ell})\defeq\prod_\calU{(H_i,\ell_i)}$ is defined as the group $H\defeq\prod_{i\in I}{H_i}$ modulo the normal subgroup $N_\calU\defeq\set{\seq{h_i}\in H}[\lim_\calU{\ell_i(h_i)}=0]$ of null sequences equipped with the (invariant) length function $\nathom{\ell}$ defined by $\nathom{\ell}(\nathom{h})\defeq\lim_\calU{\ell_i(h_i)}$, where $h=\seq{h_i}$ is a representative.
\end{definition}

Here it is important that length functions are $[0,\infty]$-valued, since otherwise $\nathom{\ell}$ would not be well-defined. Some authors use a slightly different definition by restricting $H$ to the sequences $\seq{h_i}\in\prod_{i\in I}{H_i}$ of uniformly bounded length, i.e.~$\sup_{i\in I}{\ell_i(h_i)}<\infty$. However, we prefer the above definition since then the ultraproduct is always a quotient of a product of the finite groups $H_i$ ($i\in I$).

Another thing we mention here is that an ultraproduct as in Definition \ref{def:met_ultraprod_fin_grps} is always a topological group, i.e.~the group operation and taking inverses are continuous with respect to the topology induced by $\nathom{\ell}$. This holds because $\nathom{\ell}$ is invariant.

Lastly, we remark that the algebraic ultraproduct of a family of finite groups is isomorphic to the metric ultraproduct of these groups equipped with the discrete length function with respect to the same ultrafilter. In this sense we can view every algebraic ultraproduct as a metric ultraproduct.

Now we can point out the announced characterization of $\calC$-approximable abstract groups via metric ultraproducts. Let us call the class $\calC$ trivial if either $\calC=\emptyset$ or $\calC=\set{\trivgrp}$. Here is the promised characterization. 

\begin{lemma}\label{lem:char_C_approx_abs_grp_via_ultraprod}
	If $\calC$ is a non-trivial class, every abstract $\calC$-approximable group $G$ is isomorphic to a discrete subgroup of a metric ultraproduct $(\nathom{H},\nathom{\ell})=\prod_\calU{(H_i,\ell_i)}$ of $\calC$-groups $H_i$ with invariant length function $\ell_i$ ($i\in I$) such that $\diam(H_i,\ell_i)=1$ and the distance between the images of any two different elements of $G$ is one. 
	If $G$ is countable, $I$ can be chosen to be $\nats$ with the natural order and $\calU$ to be some non-principal ultrafilter.
	
	Conversely, any subgroup of a metric ultraproduct of $\calC$-groups with invariant length function is $\calC$-approximable as an abstract group.
\end{lemma}

The proof of this result is identical to the corresponding proof in the sofic case, which is well known. Hence we omit it here.

\subsection{On $\calC$-approximable topological groups}

In view of Lemma \ref{lem:char_C_approx_abs_grp_via_ultraprod} it is natural to generalize the notion of a $\calC$-approximable group to topological groups using ultraproducts:

\begin{definition}\label{def:C_approx_top_grp}
	A topological group is called \newnotion{$\calC$-approximable} if it embeds continuously into a metric ultraproduct of $\calC$-groups with invariant length functions.
\end{definition}

Lemma \ref{lem:char_C_approx_abs_grp_via_ultraprod} indicates the following class of examples of $\calC$-approximable topological groups:

\begin{example}
	Every $\calC$-approximable abstract group equipped with the discrete topology is $\calC$-approximable as a topological group. 
\end{example}

Conversely, a $\calC$-approximable topological group is $\calC$-approximable as an abstract group when we \enquote*{forget} its topology.

To present more classes of examples, we need an auxiliary result. The following lemma gives a sufficient condition for a metric group to be isomorphic to an ultraproduct of finite metric groups. Its proof is trivial.

\begin{lemma}\label{lem:met_grp_iso_ultprod_fin_met_grps}
	Let $(G,\ell_G)$ be a group with invariant length function, $I$ an index set, $\calU$ an ultrafilter on $I$, $\seq{K_i,\ell_i}_{i\in I}$ a sequence of finite groups with invariant length function, and $(\nathom{K},\nathom{\ell})\defeq\prod_\calU{(K_i,\ell_i)}$ its metric ultraproduct.
	\begin{enumerate}[(i)]
		\item Assume there are mappings $\varphi_i:G\to K_i$, which are isometric and a homomorphism in the $\calU$-limit, i.e.~
		$$
		\lim_\calU{d_i(\varphi_i(g),\varphi_i(h))}=d_G(g,h) \text{ and } \lim_\calU{d_i(\varphi_i(g)\varphi_i(h),\varphi_i(gh))}=0.
		$$ 
		Then there is an isometric embedding $\nathom{\varphi}:(G,\ell_G)\into(\nathom{K},\nathom{\ell})$ in the ultraproduct defined by $\nathom{\varphi}(g)\defeq\nathom{(\varphi_i(g))}$.
		\item The embedding $\nathom{\varphi}$ is surjective if and only if for every $\seq{k_i}\in K\defeq\prod_{i\in I}{K_i}$ there exists $g\in G$ such that $\lim_\calU{d_i(\varphi_i(g),k_i)}=0$. 
		\item It surjects onto the subgroup of elements of finite length of $(\nathom{K},\nathom{\ell})$ if the previous assertion holds for all $\seq{k_i}\in K$ with $\sup_{i\in I}{\ell_i(k_i)}<\infty$.
	\end{enumerate}
\end{lemma}

Let $\calC^\rmP$ and $\calC^{\mathrm{SP}}$ be the class of finite products of $\calC$-groups and the class subgroups of finite products of $\calC$-groups, respectively. Now, we investigate which profinite groups are $\calC$-approximable as topological groups.  The standard example is the given by the following lemma:

\begin{lemma}\label{lem:ctbl_prod_C_grps_C^P_approx}
	Let $H_i$ ($i\in\nats_{>0}$) be $\calC$-groups. Then the profinite group $P\defeq\prod_{i\in\nats}{H_i}$ is isomorphic to a metric ultraproduct of $\calC^\rmP$-groups and so $\calC^\rmP$-approximable as a topological group.
\end{lemma}

\begin{proof}
	We want to apply (i) and (ii) of Lemma \ref{lem:met_grp_iso_ultprod_fin_met_grps} to $G\defeq P$. Equip $G$ with the invariant length function $\ell_G(h)\defeq\max\set{1/i}[i\in\nats_{>0},\,h_i\neq 1_{H_i}]\setjoin\set{0}$, where $h=\seq{h_i}$.
	Let $I\defeq\nats_{>0}$ and $\calU$ be some non-principal ultrafilter on $\nats$. Set $K_i\defeq H_1\times\cdots\times H_i\subgrpeq G$ and let $\ell_i$ be the restriction of $\ell_G$ to $K_i$. Define $\varphi_i:G\to K_i$ in such a way that for every $g\in G$ the distance $d_G(\varphi_i(g),g)$ is minimal. By definition of $\ell_G$ we have that $d_G(\varphi_i(g),g)<1/i$. Hence it is easy to verify that condition (i) of Lemma \ref{lem:met_grp_iso_ultprod_fin_met_grps} is fulfilled. For (ii) define $g$ as the $\calU$-limit of $\seq{k_i}\in K$ (which exists by compactness). Then $\lim_\calU{d_i(\varphi_i(g),k_i)}\leq\lim_\calU{d_G(\varphi_i(g),g)}+\lim_\calU{d_G(g,k_i)}=0$. This ends the proof.
\end{proof}

From the previous example we derive the following result.

\begin{lemma}\label{lem:char_approx_profin_grps}
	 For a pro-$\calC^{\mathrm{SP}}$ group $P$ the following are equivalent:
	\begin{enumerate}[(i)]
		\item $P$ is $\calC^\rmP$-approximable as a topological group.
		\item $P$ is metrizable.
		\item $P$ is first-countable.
		\item $P$ is the inverse limit of a countable inverse system of $\calC^{\mathrm{SP}}$-groups with all maps surjective.
		\item $P$ is a closed topological subgroup of a countable product of $\calC$-groups.
	\end{enumerate}
\end{lemma}

\begin{proof}
	The implications (i)$\implies$(ii)$\implies$(iii) are trivial. (iii)$\implies$(iv): Let $\calB$ be a countable system of open neighborhoods at $1_G$. For each $B\in\calB$ we can find an open normal subgroup $N\subseteq B$ such that $P/N$ is a subgroup of a $\calC^{\mathrm{SP}}$-group, so itself a $\calC^{\mathrm{SP}}$-group (by Proposition 0.3.3(a) and Proposition 1.2.1 of \cite{wilson1998profinite}). Let $\calN$ be the collection of these subgroups. Since $\bigsetmeet\calB=\set{1_P}$, as $P$ is Hausdorff, the same holds for $\calN$. Moreover, for $M,N\in\calN$ we have $P/(M\setmeet N)\subgrpeq P/M\times P/N$, so $P/(M\setmeet N)$ is a $\calC^{\mathrm{SP}}$-group too. 
	Hence we may assume that $\calN$ is closed for finite intersections and apply Proposition 1.2.2 of \cite{wilson1998profinite} to obtain that $P$ is the inverse limit of the $\calC^{\mathrm{SP}}$-groups $P/N$ ($N\in\calN$) with respect to the natural maps $P/M\to P/N$ for $M\subgrpeq N$ ($M,N\in\calN$).
	
	(iv)$\implies$(v): By the standard construction of the inverse limit, it embeds into a countable product of $\calC^{\mathrm{SP}}$-groups, which (by definition) embeds into a countable product of $\calC$-groups. For 
	(v)$\implies$(i) we only need to show that a countable product of $\calC$-groups is $\calC^\rmP$-approximable. This is Lemma \ref{lem:ctbl_prod_C_grps_C^P_approx}. The proof is complete.
\end{proof}

\begin{remark}
	The previous lemma implies that if a pro-$\calC^{\mathrm{SP}}$ group embeds continuously into a  metric ultraproduct of $\calC^\rmP$-groups with invariant length function, then it already embeds into such an ultraproduct of countably many groups.
\end{remark}

We are now able to present the following important example: 

\begin{example}
		If $P=\closure{\gensubgrp{x_1,\ldots,x_r}}$ is a topologically finitely generated pro-$\calC^{\mathrm{SP}}$ group, then $P$ is $\calC^\rmP$-approximable as a topological group. 
\end{example}

\begin{proof}
	Indeed, $P$ embeds continuously into the product of all its continuous finite quotients $\prod_N{P/N}$ and finite generation implies that there are only countably many of these. By Proposition 1.2.1 of \cite{wilson1998profinite} we can restrict this map to a product of subgroups of $\calC^{\mathrm{SP}}$-groups (which are itself $\calC^{\mathrm{SP}}$-groups) such that it still is an embedding. But, the latter embeds into a countable product of $\calC$-groups. Hence $P$ is $\calC^\rmP$-approximable (by Lemma \ref{lem:char_approx_profin_grps}).
\end{proof}

However, it is also simple to find examples of profinite groups that are not approximable by finite groups:

\begin{example}
	Uncountable products of (non-trivial) finite groups are not metrizable and hence not approximable by finite groups.
\end{example}

Now we turn to Lie groups. The following example demonstrates that connected abelian Lie groups can always be approximated by finite abelian groups in the sense of Definition \ref{def:C_approx_top_grp}.
Henceforth, let $\catAb_d$ be the class of finite abelian groups which are a direct sum of at most $d$ cyclic groups.

\begin{lemma}\label{lem:ctd_ab_Lie_grp_Ab_approx}
	Every $d$-dimenional connected abelian Lie group $L=\reals^m\times(\reals/\ints)^n$ ($m+n=d$) equipped with the \enquote*{euclidean} length function $\ell_L$ is isometrically isomorphic to the subgroup of elements of finite length of an ultraproduct of $\catAb_d$-groups with length function and hence $\catAb_d$-approximable.
\end{lemma}

\begin{proof}
	We wish to apply (i) and (iii) of Lemma \ref{lem:met_grp_iso_ultprod_fin_met_grps} to $G\defeq L$ with euclidean length function $\ell_G\defeq\ell_L$. Let $I\defeq\nats$ and $\calU$ be some non-principal ultrafilter on $\nats$. For $i\in\nats_{>0}$ set 
	$$
	S_i\defeq\lrset{\frac{-i^2}{i},\frac{-i^2+1}{i},\ldots,\frac{i^2}{i}}^m\times\left(\frac{1}{i}\ints/\ints\right)^n,
	$$ 
	and $K_i\defeq (\ints/\genidl{4i^2})^m\times(\ints/\genidl{i})^n$. Define $\alpha_i:\set{-i^2/i,(-i^2+1)/i,\ldots,i^2/i}^m\to(\ints/\genidl{4i^2})^m$ by $x\mapsto \nathom{ix}$, let $\beta_i:(\frac{1}{i}\ints/\ints)^n\to(\ints/\genidl{i})^n$ be the canonical isomorphism, and set  $\gamma_i:S_i\to K_i$ to be the map $(x,y)\mapsto(\alpha_i(x),\beta_i(y))$. Moreover, equip $K_i$ with the unique length function that turns $\gamma_i$ into an isometry. Let $\delta_i:G\to S_i$ be a map such that $d_G(\delta_i(g),g)$ is minimal for all $g\in G$. Now define $\varphi_i\defeq\gamma_i\compose\delta_i$. Clearly,  $d_G(\varphi_i(g),g)$ tends to zero for all $g\in G$. Hence condition (i) of Lemma \ref{lem:met_grp_iso_ultprod_fin_met_grps} holds. Condition (iii) follows from compactness of closed balls of finite radius in $G$ by the same argument as at the end of the proof of Lemma \ref{lem:ctbl_prod_C_grps_C^P_approx}. The proof is complete.
\end{proof}

We will see in Theorem \ref{thm:ctd_Fin_approx_Lie_grps_ab} of Section \ref{sec:approx_Lie_grps} that connected abelian Lie groups are the only $\catFin$-approximable  connected Lie groups.

\section{On $\catSol$-approximable groups}\label{sec:Sol_approx_grps}

Subsequently, let $\catSol$ (resp.~$\catNil$) be the class of finite solvable (resp.~nilpotent) groups. In this section we establish the following theorem. 

\begin{theorem}\label{thm:fin_gen_perf_grp_not_C_approx}
	Any non-trivial finitely generated and perfect group is not $\catSol$-approximable.
\end{theorem}

As a consequence, a finite group is $\catSol$-approximable if and only if it is solvable: Indeed, any finite solvable group is $\catSol$-approximable. On the other hand, a non-solvable finite group contains a non-trivial perfect subgroup and hence cannot be $\catSol$-approximable by Remark \ref{rmk:C_approx_iff_fin_gen_subgrp_C_approx} and Theorem \ref{thm:fin_gen_perf_grp_not_C_approx}.

Initially, Howie proved in \cite{howie1984the} that the group $\grp{x,y}[x^{-2}y^{-3},x^{-2}(xy)^5]$ is not $\catNil$-approximable. We mimic his proof for any non-trivial finitely generated perfect group and then extend it by establishing that these groups are not even $\catSol$-approximable using techniques of Segal \cite{segal2000closed,segal2009words}.

In preparation of the proof of Theorem \ref{thm:fin_gen_perf_grp_not_C_approx}, we need an auxiliary result from \cite{glebskyrivera2008sofic}. Recall that the pro-$\calC$ topology on a group $K$ is the initial topology induced by all homomorphisms to $\calC$-groups equipped with the discrete topology. Hence the closure $\closure{S}$ of a subset $S\subseteq K$ in this topology is characterized as follows: An element $k\in K$ lies in this closure if and only if $\varphi(k)\in\varphi(S)$ for all homomorphism $\varphi:K\to H$, where $H$ is a $\calC$-group.

Adapting Theorem 4.3 of \cite{glebskyrivera2008sofic}, one can prove the following theorem relating $\calC$-approximable groups to the pro-$\calC$ topology on a free group of finite rank:

\begin{theorem}\label{thm:char_C_approx_grps_pro_C}
	Let $\freegrp/N$ be a presentation of a group $G$, where $\rk(\freegrp)<\infty$.
	Then, if $G$ is $\calC$-approximable, for each finite sequence $n_1,\ldots,n_k\in N$ it holds that $\closure{n_1^\freegrp\cdots n_k^\freegrp}\subseteq N$ (in the pro-$\calC$ topology on $\freegrp$). The converse holds if $\calC$ is closed with respect to finite products and subgroups.
\end{theorem}

\begin{remark}\label{rmk:abs_res_C_grps_C_approx}
	If $\calC$ is closed with respect to finite products and subgroups, the previous theorem implies that residually $\calC$-groups are $\calC$-approximable as abstract groups, since if $G=\freegrp/N$ is a finitely generated such group, for each finite sequence $n_1,\ldots,n_k\in N$ we obtain $\closure{n_1^\freegrp\cdots n_k^\freegrp}\subseteq \closure{N}=N$ (in the pro-$\calC$ topology on $\freegrp$).
\end{remark}

In view of Theorem \ref{thm:char_C_approx_grps_pro_C}, when $\calC$ is closed for subgroups, to prove the existence of a non-$\calC$-approximable group, it suffices to find a normal subgroup $N\norsubgrpeq\freegrp$ of a free group of finite rank, an element $x\in\freegrp\setminus N$, and a sequence $n_1,\ldots,n_k\in N$ such that $\varphi(x)\in \varphi(n_1)^H\cdots\varphi(n_k)^H$ for any surjective homomorphism $\varphi:\freegrp\to H$ to a $\calC$-group.  

As both classes $\catNil$ and $\catSol$ are closed with respect to subgroups, we shall construct a situation as described before. 

Subsequently, let $\freegrp$ be freely generated by $x_1,\ldots,x_r$. Fix a presentation $\freegrp/N$ of some non-trivial perfect group $P$, and an element $x\in\freegrp\setminus N$. The assumption that $P$ is perfect is equivalent to the fact that $\freegrp=\freegrp'N$. Hence we can find $n_1,\ldots,n_r,n\in N$ such that $x_i\con n_i$ ($i=1,\ldots,r$) and $x\con n$ modulo $\freegrp'$. Consider a surjective homomorphism $\varphi:\freegrp\to H$ to some finite group $H$ (later $H$ will be assumed to be nilpotent resp.~solvable). Writing $y_i\defeq \varphi(x_i)$, $h_i\defeq\varphi(n_i)$ ($i=1,\ldots,r$), $y\defeq\varphi(x)$, and  $h\defeq\varphi(n)$, the above translates to $y_i\con h_i$ ($i=1\ldots,r$) and $y\con h$ modulo $H'$. Clearly, $h_1,\ldots,h_r$ generate $H$ modulo $H'$ (as $\varphi$ is surjective). Now we need a lemma. To state it, it becomes necessary to introduce some notation. In a group $G$ define the commutator of two elements $g,h\in G$ by $\commutator{g,h}\defeq g^{-1}h^{-1}gh=g^{-1}g^h$, for $S\subseteq G$ and $g\in G$ write $\commutator{S,g}$ for the set $\set{\commutator{s,g}}[s\in S]$, and for subgroups $K,L\subgrpeq G$ write $\commutator{K,L}$ for the subgroup generated by $\set{\commutator{k,l}}[k\in K,\, l\in L]$.

\begin{lemma}[Proposition 1.2.5 from \cite{segal2009words}]\label{lem:com_gen_id_wrt_lcs}
	Let $L\norsubgrpeq G$ be groups, and suppose that $G=G'\gensubgrp{g_1,\ldots,g_m}$. Then
	$$
	\commutator{L,G}=\commutator{L,g_1}\cdots\commutator{L,g_m}\commutator{L,_l G}
	$$
	for all $l\geq 1$. Here, $\commutator{L,_l G}$ denotes the subgroup $\commutator{L,\underbrace{G,\ldots,G}_{l}}$ of $G$.
\end{lemma}

\begin{proof}[Proof of Theorem \ref{thm:fin_gen_perf_grp_not_C_approx}, Part 1]
We apply the previous lemma to $G\defeq L\defeq H$, $m\defeq r$, and $g_i\defeq h_i$ ($i=1,\ldots,r$). Moreover, we choose $l\geq 1$ to be an integer such that $\lcssubgrp_l(H)=\lcssubgrp_\omega(H)$ (here $\lcssubgrp_l(H)$ is the $l$th term in the lower central series of $H$ and $\lcssubgrp_\omega(H)=\bigsetmeet\set{\lcssubgrp_l(H)}[l\in\nats_{>0}]$). Hence there exist $l_{ij},l_j\in H$ ($i,j=1\ldots,r$) such that $y_i \con h_i\commutator{l_{i1},h_1}\cdots\commutator{l_{ir},h_r}$ ($i=1,\ldots,r$) and $y\con h\commutator{l_1,h_1}\cdots\commutator{l_r,h_r}$ modulo $\lcssubgrp_\omega(H)$. 
Assuming $H$ is nilpotent (so $\lcssubgrp_\omega(H)=\trivgrp$), the last congruence shows that
$$
y=\varphi(x)\in \varphi(n_1')^H\cdots\varphi(n'_{k_{\catNil}})^H,
$$
where $k_{\catNil}=2r+1$ and $\seq{n'_j}_{j=1}^{k_{\catNil}}$ is a fixed sequence with entries in the set $\set{n,n_1^{\pm 1},\cdots,n_r^{\pm 1}}$. 
Thus $P$ cannot be $\catNil$-approximable.
\end{proof}
To prove that $P$ is not $\catSol$-approximable, we need the following deeper result of Segal on finite solvable groups:

\begin{theorem}[Theorem 2.1 from \cite{segal2000closed}]\label{thm:com_gen_id_sol_grp}
	Assume $G$ is a finite solvable group and $\lcssubgrp_\omega(G)\gensubgrp{g_1,\ldots,g_m}=G$ for some $m\in\nats$. Moreover, assume that $G$ is generated by $d$ elements. Then there is a fixed sequence $\seq{i_j}_{j=1}^{m'}$ of indices in $\set{1,\ldots,m}$, whose entries and length $m'$ only depend on $d$ and $m$ such that 
	$$
	\lcssubgrp_\omega(G)=\prod_{j=1}^{m'}\commutator{\lcssubgrp_\omega(G),g_{i_j}}.
	$$
\end{theorem}
\begin{proof}[Proof of Theorem \ref{thm:fin_gen_perf_grp_not_C_approx}, Part 2]
Assume that $H$ is solvable. We want to apply Theorem \ref{thm:com_gen_id_sol_grp} to $G\defeq H$.
Since $\varphi$ is surjective, the elements $y_1=\varphi(x_1),\ldots,y_r=\varphi(x_r)$ generate $H$, so we may set $d\defeq r$. We still have to define the elements $g_1,\ldots,g_m\in G$. From the above congruences we conclude that the sequence 
$$
h_1,\ldots,h_r,(h_1^{-1})^{l_{11}},\cdots,(h_r^{-1})^{l_{1r}},\cdots,(h_1^{-1})^{l_{r1}},\cdots,(h_r^{-1})^{l_{rr}}
$$
is a good choice for $g_1,\ldots,g_m$. Thus $m\defeq r(r+1)$ is bounded in terms of $r$.

The theorem gives us, similarly as in the nilpotent case, a fixed sequence $\seq{n''_j}_{j=1}^{k_{\catSol}}$ with entries in $\set{n,n_1^{\pm 1},\cdots,n_r^{\pm 1}}$, whose length $k_{\catSol}=k_{\catNil}+2m'$ is bounded in terms of $r$, such that
$$
y=\varphi(x)=\varphi(n_1'')^H\cdots\varphi(n_{k_{\catSol}}'')^H.
$$
Thus $P$ cannot be $\catSol$-approximable.
\end{proof}
Note that finite generation is crucial here. Indeed, there exist countably infinite locally finite-$p$ groups which are perfect and even characteristically simple \cite{mclain1954characteristically}. By Example \ref{exl:loc_C_impl_C_approx}, these groups are $\catNil$-approximable (since finite $p$-groups are nilpotent), but by definition they are not finitely generated. It is known that locally finite-solvable groups cannot be non-abelian simple \cite[p.~154]{robinson1972finiteness}, but it seems to be an open problem if there exist $\catSol$-approximable simple groups.

\section{On $\catFin$-approximable groups}\label{sec:Fin_approx_grps_hom_PSL}

Let $\catPSL$ be the class of simple groups of type $\PSL_n(q)$, i.e.~$n\in\nats_{\geq 2}$ and $q$ is a prime power and $(n,q)\neq (2,2),(2,3)$, and recall that $\catFin$ is the class of all finite groups.
In this section we prove the following result.

\begin{theorem}\label{thm:smpl_Fin_approx_grp_ultraprod_smpl_grps}
	Any non-trivial finitely generated $\catFin$-approximable group has a non-trivial $\catPSL$-approximable quotient. In particular, every simple $\catFin$-approximable group is $\catPSL$-approximable.
\end{theorem}

To prove Theorem \ref{thm:smpl_Fin_approx_grp_ultraprod_smpl_grps} we need some preparation. At first we recall a classical lemma of Goursat \cite{goursat1889substitutions}:

\begin{lemma}[Goursat's Lemma]
	Let $G\subgrpeq K\times L$ be a subdirect product, i.e.~the restricted projection maps $\pi_K:G\to K$, $\pi_L:G\to L$ are surjective. Set $M\defeq\ker(\pi_L)$ and $N\defeq\ker(\pi_K)$. Then $M\norsubgrpeq K$, $N\norsubgrpeq L$, and the image of $G$ in $K/M\times L/N$ is the graph of an isomorphism.
\end{lemma}

We need the preceding lemma for the following auxiliary result. Recall that a profinite group is called semisimple if it is the direct product of finite non-abelian simple groups. Moreover, a finite group $G$ is \newnotion{almost simple} if it has a unique minimal normal subgroup $N$ which is non-abelian simple; in this case $N\norsubgrpeq G\subgrpeq\Aut(N)$.

\begin{lemma}\label{lem:prod_alm_smpl_grps}
	Let $G$ be a closed subdirect product of a profinite group $A=\prod_{i\in I}{A_i}$, where $A_i$ is almost simple ($i\in I$). Then $G$ contains a closed normal semisimple subgroup $H$ such that $G/H$ is solvable of derived length at most three and each simple factor of $H$ is normal in $G$. 
\end{lemma}

\begin{proof}
	For $J\subseteq I$ let $\pi_J:A\to\prod_{j\in J}{A_j}$ be the projection maps. Then by Proposition 1.2.2 of \cite{wilson1998profinite} $G$ is the inverse limit of the groups $\pi_J(G)$ ($J\subseteq I$ finite) together with the natural maps $\pi_J(G)\to\pi_{J'}(G)$ for $J'\subseteq J$.
	
	Using Goursat's Lemma one can show by induction on $\card{J}$ that for $J\subseteq I$ finite there exist $r\in\nats$ and finite non-abelian simple groups $S_1,\ldots,S_r$ such that $S_1\times\cdots\times S_r\norsubgrpeq\pi_J(G)\subgrpeq\Aut(S_1)\times\cdots\times\Aut(S_r)$. In this situation, for $j_0\in I\setminus J$ the projection $\pi_{J\setjoin\set{j_0}}(G)\to\pi_J(G)$ either is an isomorphism or there exists a finite non-abelian simple group $S_{r+1}$ such that $S_1\times\cdots\times S_{r+1}\norsubgrpeq\pi_{J\setjoin\set{j_0}}(G)\subgrpeq\Aut(S_1)\times\cdots\times\Aut(S_{r+1})$ and the restriction of $\pi_{J\setjoin\set{j_0}}(G)\to\pi_J(G)$ to the socle $S_1\times\cdots S_{r+1}$ of $\pi_{J\setjoin\set{j_0}}$ is the natural projection onto $S_1\times\cdots\times S_r$ (the socle of $\pi_J(G)$).
	
	Now it is clear that $G$, as the inverse limit of the groups $\pi_J(G)$ ($J\subseteq I$ finite) and the maps $\pi_J(G)\to\pi_{J'}(G)$, contains the inverse limit $H$ of the socles of these groups together with the restricted maps. It is routine to check that $H$ has the desired properties. The fact that $G/H$ is solvable of derived length at most three is implied by Schreier's conjecture.
\end{proof}

Now we can start with the proof of Theorem \ref{thm:smpl_Fin_approx_grp_ultraprod_smpl_grps}: We will prove that our group is $\catPSL$-approximable where we endow the groups $\PSL_n(q)$ with the conjugacy metric -- see Equation \eqref{conjmetric} for a definition.

If the group in the theorem is not perfect, it has a non-trivial cyclic quotient which clearly has the desired property.
So let $P=\freegrp/N$ be perfect, where $\freegrp$ is freely generated by $x_1,\ldots,x_r$ and $N\norsubgrpeq\freegrp$. Let $\completion{\freegrp}$ be the profinite completion of $\freegrp$ and $M\defeq\gennorsubgrp{N}_{\completion{\freegrp}}$ be the normal closure of $N$ in $\completion{\freegrp}$. Identifying $\freegrp$ with its image in the profinite completion, it follows from Theorem \ref{thm:char_C_approx_grps_pro_C} that $P$ is $\catFin$-approximable if and only if $N=M\setmeet\freegrp$, since for a sequence $n_1,\ldots,n_k\in N$ we have 
$$
\closure{n_1^\freegrp\cdots n_k^\freegrp}=\freegrp\setmeet n_1^{\completion{\freegrp}}\cdots n_k^{\completion{\freegrp}},
$$
where the closure on the left is taken in $\freegrp$. This is equivalent to saying that the map $\freegrp\to\completion{\freegrp}/M$ induces an embedding of $P$ in $\completion{\freegrp}/M$.

For a profinite group $G$ set $G_0\defeq\bigsetmeet\set{O\opnorsubgrpeq G}[G/O\text{ is almost simple}]$.

\begin{claim}\label{clm:1}
	It holds that $\freegrp\not\subgrpeq \completion{\freegrp}_0 M$.
\end{claim}

\begin{proof}
	Assume the contrary. Then by perfectness of $P$ there are $y_i\in N$ such that $x_i\freegrp'=y_i\freegrp'$ ($i=1,\ldots,r$). By assumption there are also $z_i\in M$ such that $x_i{\completion{\freegrp}}_0=z_i{\completion{\freegrp}}_0$ ($i=1,\ldots,r$). Set $Y\defeq\set{y_1^{\pm 1},\ldots,y_r^{\pm 1},z_1^{\pm 1},\ldots,z_r^{\pm 1}}$. As ${\completion{\freegrp}}_0$ is closed, by definition, we have that 
	$$
	\completion{\freegrp}=\closure{{\completion{\freegrp}}_0 \gensubgrp{Y}}={\completion{\freegrp}}_0\closure{\gensubgrp{Y}}
	\quad\text{and}\quad
	\completion{\freegrp}=\closure{\freegrp'\gensubgrp{Y}}=\closure{\completion{\freegrp}'\gensubgrp{Y}},
	$$
	where all closures are taken in $\completion{\freegrp}$. Hence by Theorem 1.7 of \cite{nikolovsegal2012generators} applied to $G\defeq\completion{\freegrp}$ and $H\defeq{\completion{\freegrp}}_0$ we get that $M\supgrpeq\gennorsubgrp{Y}_{\completion{\freegrp}}\supgrpeq\commutator{\completion{\freegrp}_0,\completion{\freegrp}}\supgrpeq\completion{\freegrp}_0'$. Since ${\completion{\freegrp}}_0 M/M=\completion{\freegrp}_0/(\completion{\freegrp}_0\setmeet M)$ is abelian by the preceding argument,
	we cannot have $\freegrp\subgrpeq{\completion{\freegrp}}_0 M$, since otherwise $P$ would be abelian. Contradiction proving the claim.
\end{proof}

Claim \ref{clm:1} implies that $P$ has a non-trivial homomorphism to $\completion{\freegrp}/\completion{\freegrp}_0 M$. Apply Lemma \ref{lem:prod_alm_smpl_grps} to $G\defeq\completion{\freegrp}/{\completion{\freegrp}}_0$ as a subdirect product of all almost simple quotients of $\completion{\freegrp}$. Let $H=K/\completion{\freegrp}_0$ be the semisimple group provided by the lemma. As $\completion{\freegrp}/K$ is solvable, we cannot have $K\subgrpeq\completion{\freegrp}_0 M$, otherwise the image of $P$ in $\completion{\freegrp}/\completion{\freegrp}_0 M$ would be trivial, contradicting Claim \ref{clm:1}.

Hence $(K\setmeet\completion{\freegrp}_0 M)/\completion{\freegrp}_0$ is a proper normal subgroup of the semisimple group $H=K/\completion{\freegrp}_0=\prod_{i\in I}{S_i}$, where $S_i$ ($i\in I$) are the simple factors, so by Theorem 5.12 of \cite{nikolovsegal2012generators} it is contained in a maximal normal subgroup $L/\completion{\freegrp}_0$ of the former. By the same result, $K/L$ is isomorphic (as an abstract group) to a metric ultraproduct of the $S_i$ with the conjugacy length function $\ell_i\defeq\ell_{S_i}^{\rmc}$ ($i\in I$). Note that in this situation $L$ is even normal in $\completion{\freegrp}$, since $\ell_i$ is left invariant under $\Aut(S_i)$ and $S_i\norsubgrpeq\completion{\freegrp}/\completion{\freegrp}_0$ by Lemma \ref{lem:prod_alm_smpl_grps} ($i\in I$).

\begin{claim}
	In this setting $\freegrp\not\subgrpeq LM$.
\end{claim}

\begin{proof}
	Otherwise $\commutator{\freegrp,K}\subgrpeq\commutator{LM,K}\subgrpeq L\commutator{M,K}=L(M\setmeet K)=L$. Here the first inclusion holds by assumption, whereas the second follows from the commutator identity $\commutator{lm,k}=\commutator{l,k}\commutator{\commutator{l,k},m}\commutator{m,k}$ for $k\in K$, $l\in L$ and $\commutator{l,k}\in L$, since $L\norsubgrpeq K$, and $\commutator{\commutator{l,k},m},\commutator{m,k}\in \commutator{K,M}=\commutator{M,K}$. The second last equality holds as $M,K\norsubgrpeq\completion{\freegrp}_0$, and the last by the choice of $L$. Hence $\completion{\freegrp}_0\commutator{\freegrp,K}\subgrpeq L$.
	
	Let $S$ be a simple factor of $H$. $\completion{\freegrp}/\completion{\freegrp}_0$ maps continuously on the finite discrete group $\Aut(S)$ via the conjugation action. The image of this map clearly contains the inner automorphism, since these are induced by $S$ itself. The elements $\nathom{x_1},\ldots,\nathom{x_r}$ generate a dense subgroup of $\completion{\freegrp}/\completion{\freegrp}_0$, which must induce all inner automorphisms of $S$ by the previous fact.
	
	As $S$ has trivial center, we have $\card{S/\centralizersubgrp_S(\nathom{x_{i_0}})}=\card{\commutator{S,\nathom{x_{i_0}}}}\geq\card{S}^{1/r}$ for some $i_0\in\set{1,\ldots,r}$. Lemma 3.5 of \cite{nikolovsegal2012generators} implies that $\prod_{i=1}^r{\commutator{S,\nathom{x_i}}\commutator{S,\nathom{x_i}^{-1}}}$ contains the normal subsets $\commutator{S,\nathom{x_i}}^S\subseteq S$ for $i=1,\ldots,r$. Since $\card{\commutator{S,\nathom{x_{i_0}}}^S}\geq\card{S}^{1/r}$, by Theorem 1.1 of \cite{liebeckshalev2001diameters} there is $e\in\nats$ only depending on $r$ such that
	$$
	\left(\prod_{i=1}^r{\commutator{S,\nathom{x_i}}\commutator{S,\nathom{x_i}^{-1}}}\right)^{\ast e}=S.
	$$
	This implies that $K\subgrpeq\completion{\freegrp}_0\commutator{\freegrp,K}$, since $e$ is independent of the simple factor $S$. But then $K\subgrpeq L$, a contradiction.
\end{proof}

From the previous claim we deduce that $P$ still has a non-trivial homomorphism to $\completion{\freegrp}/LM$. Since $\completion{\freegrp}/KM$ is solvable as a quotient of $\completion{\freegrp}/K$, this homomorphism restricts to $KM/LM$, which is a non-trivial homomorphic image of the metric ultraproduct $K/L$. Since the latter is simple by Proposition 3.1 of \cite{stolzthom2014lattice}, we are only left to show that $K/L$, which is a metric ultraproduct of the sequence $\seq{S_i}_{i\in I}$ of finite simple groups from above with conjugacy length function with respect to some ultrafilter $\calU$, embeds into a metric ultraproduct of groups $\PSL_{n_i}(q_i)$ equipped with the conjugacy length function $\ell^{\rm c}_i$ ($i\in I$), since then $P$ would have the same property.
	
Let us briefly sketch the argument for this: Firstly, if the limit of the ranks of the groups $S_i$ ($i\in I$) is bounded along the ultrafilter $\calU$ (where
the rank of the alternating group $A_n$ is defined to be $n$ and the sporadic groups are also considered as groups of bounded rank) the resulting ultraproduct will be a simple group
of Lie type over a pseudofinite field $k$ or an alternating group $A_n$, respectively. In the first case it clearly embeds into $\PSL_n(k)$ for $n\in\nats$ appropriately chosen. However, the latter is a metric ultraproduct of groups $\PSL_n(q_i)$ with conjugacy length function $\ell_i^{\rm c}$ ($i\in I$) for some sequence $\seq{q_i}_{i\in I}$ of prime powers. The second case is similar.
	
Hence we may assume that our ultraproduct does not involve finite simple groups from families of bounded rank.
	
We can further assume that it contains no alternating groups as we can replace any alternating group $A_n$ by $\PSL_n(q)$ for $q=p^e$ a power of a prime with $p>n$. Namely, the natural embedding $A_n\into\PSL_n(q)$, where $\PSL_n(q)$ is equipped with the projective rank length function $\ell_n^{\rm pr}$, induces the Cayley length function $\ell_n^{{\rm Cay},\set{\tau}}$ on $A_n$ with respect to the conjugacy class of a transposition of the ambient symmetric group $S_n$. The latter is Lipschitz equivalent to the conjugacy class length function by Theorem 2.15 of \cite{stolzthom2014lattice}.
	
Hence we can assume that all groups $S_i$ ($i\in I$) are classical Chevalley or Steinberg groups. But it follows from Lemmas 5.4, 6.4, and the end of Section 7 and Theorem 1.1 of \cite{liebeckshalev2001diameters} that the conjugacy length function and the projective rank function (coming from a natural embedding in some $\PSL_n(q)\subgrpeq\PGL_n(q)$) on such groups are also Lipschitz equivalent.
	
Hence we can embed our ultraproduct $K/L$ into an ultraproduct of groups $\PSL_{n_i}(q_i)$ equipped with the projective length function $\ell^{\rm pr}_i$ ($i\in I$). But by the former Lipschitz equivalence, $\ell^{\rm pr}_i$ can be replaced by the conjugacy length function $\ell_i^{\rm c}$ ($i\in I$). This ends the proof.

\section{On the approximability of Lie groups}\label{sec:approx_Lie_grps}

In this section we utilize the following theorem of the first author and Segal to deduce two results concerning the approximability of Lie groups by finite groups and one result on compactifications of pseudofinite groups.

\begin{theorem}[Theorem 1.2 of \cite{nikolovsegal2012generators}]\label{thm:gen_com_fin_grp}
	Let $g_1,\ldots,g_m$ be a symmetric generating set for the finite group $G$. If $K\norsubgrpeq G$, then
	$$
	\commutator{K,G}={\left(\prod_{j=1}^m{\commutator{K,g_j}}\right)}^{\ast e},
	$$
	where $e$ only depends on $m$.
\end{theorem}

\begin{remark}
It was remarked in \cite{glebskyrivera2008sofic} that it was an open problem at the time of writing to decide whether a finite product of conjugacy classes in a non-abelian free group is always closed in the profinite topology. 

It is a rather straightforward consequence of Theorem \ref{thm:gen_com_fin_grp} that this is not the case. Indeed, the theorem implies that in $\freegrp=\grp{x_1,\ldots,x_m}$ the profinite closure of the product of the $2me$ conjugacy classes of $x_1^{-1},x_1,\ldots,x_m^{-1},x_m$ contains the entire commutator subgroup, but it is a well known fact (see Theorem 3.1.2 of \cite{segal2009words}) that the commutator width in this group is infinite if $m>1$. 

This implication was first observed by Segal and independently discovered by Gismatullin.
\end{remark}

Actually, we shall use the following immediate corollary of Theorem \ref{thm:gen_com_fin_grp}:

\begin{corollary}\label{cor:two_gen_com_fin_grp}
	Let $G$ be a quotient of a product of finite groups, then for $g,h\in G$ and $N\in\nats$ we have
	$$
	\commutator{g^N,h^N}\in {\left(\commutator{G,g}\commutator{G,g^{-1}}\commutator{G,h}\commutator{G,h^{-1}}\right)}^{\ast e}
	$$
	for some fixed constant $e\in\nats$.
\end{corollary}

Recall that $\catFin$ denotes the class of all finite groups. At first we prove the following theorem:

\begin{theorem}\label{thm:ctd_Fin_approx_Lie_grps_ab}
	A connected Lie group is $\catFin$-approximable as a topological group if and only if it is abelian.
\end{theorem}

By Lemma \ref{lem:ctd_ab_Lie_grp_Ab_approx} we already know that connected abelian Lie groups are $\catFin$-approximable. So we are only left to prove that a $\catFin$-approximable connected Lie group is actually abelian. This will be a consequence of the following auxiliary result:

\begin{lemma}\label{lem:two_one_par_subgrps_to_ultprod_comm}
	Let $\varphi,\psi:\reals\to(\nathom{H},\nathom{\ell})=\prod_{\calU}{(H_i,\ell_i)}$ be continuous homomorphisms into a metric ultraproduct of finite groups $H_i$ with invariant length function $\ell_i$ ($i\in I$). Then for all $s,t\in\reals$ it holds that $\commutator{\varphi(s),\psi(t)}=1_{\nathom{H}}$.
\end{lemma}

Let us first prove Theorem \ref{thm:ctd_Fin_approx_Lie_grps_ab} using Lemma \ref{lem:two_one_par_subgrps_to_ultprod_comm}.

\begin{proof}[Proof of Theorem \ref{thm:ctd_Fin_approx_Lie_grps_ab}]
	Assume $L$ is a connected $\catFin$-approximable Lie group. Then there is an embedding $\iota:L\into(\nathom{H},\nathom{\ell})=\prod_\calU{(H_i,\ell_i)}$ into a metric ultraproduct of finite groups with invariant length function. If $a,b\in L$ are in the image of the exponential map, Lemma \ref{lem:two_one_par_subgrps_to_ultprod_comm} implies that $\iota(a)$ and $\iota(b)$ commute. So as $\iota$ is injective, $a$ and $b$ commute. Hence by connectedness $L=L^0$ is abelian. This ends the proof.
\end{proof}

We are still left to prove Lemma \ref{lem:two_one_par_subgrps_to_ultprod_comm}:

\begin{proof}[Proof of Lemma \ref{lem:two_one_par_subgrps_to_ultprod_comm}]
	For $\varepsilon>0$ by continuity we can choose $N\in\nats_{>0}$ large enough such that $\nathom{\ell}(\varphi(s/N)),\nathom{\ell}(\psi(t/N))<\varepsilon$.
	Set $G\defeq\nathom{H}$, $g\defeq\varphi(s/N)$ and $h\defeq\psi(t/N)$ and apply Corollary \ref{cor:two_gen_com_fin_grp}. This gives 
	$$
	\commutator{\varphi(s),\psi(t)}=\commutator{g^N,h^N}\in{\left(\commutator{\nathom{H},g}\commutator{\nathom{H},g^{-1}}\commutator{\nathom{H},h}\commutator{\nathom{H},h^{-1}}\right)}^{\ast e},
	$$
	whence $\nathom{\ell}(\commutator{\varphi(s),\psi(t)})<8e\varepsilon$ by invariance of $\nathom{\ell}$ and the triangle inequality. Since $\varepsilon>0$ was arbitrary, the proof is complete.
\end{proof}

Note that Theorem \ref{thm:ctd_Fin_approx_Lie_grps_ab} provides an answer to Question 2.11 of Doucha \cite{doucha2016metric} whether there are groups with invariant length function that do not embed in a metric ultraproduct of finite groups with invariant length function.
Since every compact Lie group can be equipped with an invariant length function that generates its topology, every such group with non-abelian identity component is an example of such a group by the previous theorem. (Indeed, Theorem \ref{thm:ctd_Fin_approx_Lie_grps_ab} does even provide topological types of groups which cannot occur as subgroups of such a metric ultraproduct.)

Before we continue with our next result, let us state the following two remarks.

\begin{remark}\label{rmk:Lie_grp_not_Fin_approx_top_matters}
	In Theorem \ref{thm:ctd_Fin_approx_Lie_grps_ab} the topology of the Lie group matters. Indeed, any linear Lie group is $\catFin$-approximable as an abstract group by Remark \ref{rmk:C_approx_iff_fin_gen_subgrp_C_approx}, since all its finitely generated subgroups are residually finite by Malcev's Theorem and hence $\catFin$-approximable by Remark \ref{rmk:abs_res_C_grps_C_approx}.
	
	Thus any linear Lie group $L$ is embeddable (as an abstract group) into a metric ultraproduct of finite groups with invariant length function indexed over, say, the partially ordered set of pairs consisting of a finite subset of $L$ and a positive rational number. We will now show that we can even choose this index set to be $\nats$. Namely, if $L \subseteq \SL_n(\complex)$, then $L$ can be embedded into the algebraic ultraproduct $\prod_\calU \SL_n(p^{p!})$, where $\calU$ is a non-principal ultrafilter on the set of prime numbers. Indeed, this ultraproduct is isomorphic to $\SL_n(k)$, where $k= \prod_{\calU} \GF(p^{p!})$ is a pseudofinite field. Now it is straightforward to see that $k$ contains the field $k_0 = \prod_{\calU} \GF(p)$ together with its algebraic closure $k':= k_0^{\rm alg}$. However, $k'$ is an algebraically closed field of characteristic zero and cardinality $2^{\aleph_0}$ (a result due to  Shelah \cite{shelah1970cardinality}) and hence isomorphic to $\complex$. Note that, if we view the above algebraic ultraproduct as a metric ultraproduct, the induced topology on $\SL_n(\complex)$ is discrete.
	
	Since some non-linear Lie groups admit finitely presented subgroups which are not residually finite \cite{deligne1978extensions}, it is clear that such embeddings cannot exist without the assumption of linearity.
\end{remark}

\begin{remark}\label{rmk:reals_not_Sym_approx}
	When one approximates with symmetric groups, one can not even embed the real line $\reals$ in a metric ultraproduct of such groups with invariant length function. E.g.~for the symmetric group $S_n$ it can be shown that all invariant length functions $\ell$ on it satisfy $\ell(\sigma^k)\leq 3\ell(\sigma)$, for every $k\in\ints$ and $\sigma\in S_n$. Using this identity, it is simple to deduce that the only continuous homomorphism of $\reals$ into a metric ultraproduct of finite symmetric groups  with invariant length function is trivial.
\end{remark}

Referring to the question of Zilber \cite[p.~17]{zilber2014perfect} (also Question 1.1 of Pillay \cite{pillay2015remarks}) whether a compact simple Lie group can be a quotient of the algebraic ultraproduct of finite groups, we present the following second application of Corollary \ref{cor:two_gen_com_fin_grp}:

\begin{theorem}\label{thm:Lie_grp_quot_prod_fin_grps}
	A Lie group equipped with an bi-invariant metric generating its topology that is an abstract quotient of a product of finite groups has abelian identity component.
\end{theorem}

The proof of this result is almost identical to the proof of Theorem \ref{thm:ctd_Fin_approx_Lie_grps_ab}.

\begin{proof}
	Let $(L,\ell_L)$ be such a Lie group with invariant length function and $a,b\in L$ be in the image of the exponential map. For $\varepsilon>0$ we find $N\in\nats_{>0}$, $g,h\in L$ such that $\ell_L(g),\ell_L(h)<\varepsilon$ and $g^N=a$, $h^N=b$. Then applying Corollary \ref{cor:two_gen_com_fin_grp} to $G\defeq L$ yields 
	$$
	\commutator{a,b}=\commutator{g^N,h^N}\in{\left(\commutator{L,g}\commutator{L,g^{-1}}\commutator{L,h}\commutator{L,h^{-1}}\right)}^{\ast e},
	$$
	whence $\ell_L(\commutator{a,b})<8e\varepsilon$ by the invariance of $\ell_L$ and the triangle inequality. This shows that $a$ and $b$ commute. Hence, as $L^0$ is generated by the image of the exponential map, it must be abelian.
\end{proof}

Theorem \ref{thm:Lie_grp_quot_prod_fin_grps} implies that any compact simple Lie group, the simplest example being $\SO_3(\reals)$, is not a quotient of a product of finite groups, answering Zilber's question (and hence also answers Question 1.1 of Pillay \cite{pillay2015remarks}).

Moreover, Theorem \ref{thm:Lie_grp_quot_prod_fin_grps} remains valid if we replace the product of finite groups by a \newnotion{pseudofinite group}, i.e.~a group which is a model of the theory of all finite groups.
It then also provides a negative answer to Question 1.2 of Pillay \cite{pillay2015remarks}, whether there is a surjective homomorphism from a pseudofinite group to a compact simple Lie group.

Before we state the last theorem of this section, we digress briefly by pointing out a further application of Theorems \ref{thm:ctd_Fin_approx_Lie_grps_ab} and \ref{thm:Lie_grp_quot_prod_fin_grps}. 

Referring to \cite{turing1938finite}, we call a compact group $G$ with compatible invariant length function $\ell_G$ Turing-approximable if for all $\varepsilon>0$ there is a finite set $S_{\varepsilon}$, a group $H_\varepsilon$, and a bijection $\gamma_{\varepsilon} \colon S_{\varepsilon}\to H_\varepsilon$ such that for all $g\in G$ there is $s\in S_\varepsilon$ with $d_G(g,s)<\varepsilon$ and $d_G(gh,\gamma_{\varepsilon}^{-1}(\gamma_{\varepsilon}(g)\gamma_{\varepsilon}(h)))
<\varepsilon$ for $g,h\in S_\varepsilon$.
Define for $g\in H_\varepsilon$
$$
\ell_{\varepsilon}(g)\defeq \card{H_\varepsilon}^{-2} \sum_{f,h\in H_\varepsilon}{d_G(\gamma_{\varepsilon}^{-1}(fgh),\gamma_{\varepsilon}^{-1}(fh))}.
$$
It is routine to check that $\ell_{\varepsilon}$ is an invariant length function on $H_\varepsilon$ and that for all $g\in H_\varepsilon$ we have
$$
\modulus{\ell_{\varepsilon}(g)-\ell_G(\gamma_{\varepsilon}^{-1}(g))}<3\varepsilon.
$$
Set $\delta_\varepsilon:G\to S_\varepsilon$ such that $d_G(\delta_\varepsilon(g),g)$ is minimal for all $g\in G$.

In this situation we can apply Lemma \ref{lem:met_grp_iso_ultprod_fin_met_grps}, setting $I\defeq\nats$, $\calU$ to be a non-principal ultrafilter on $\nats$, $K_i\defeq H_{1/i}$, and $\varphi_i\defeq\gamma_{1/i}\compose\delta_{1/i}$. Again one checks easily that we may apply (i) and (ii) of the lemma. Hence a Turing-approximable group is isomorphic to a metric ultraproduct of finite groups with invariant length function. Thus Theorem \ref{thm:ctd_Fin_approx_Lie_grps_ab} as well as Theorem \ref{thm:Lie_grp_quot_prod_fin_grps} imply that a Turing-approximable Lie group has abelian identity component. This is the main result of \cite{turing1938finite}. By Lemma 3.4 of \cite{gelander2012limits} the latter condition is also sufficient for a compact Lie group to be
Turing-approximable.

Let us now turn to pseudofinite groups. By a \newnotion{compactification} of an abstract group $G$, we mean a compact group $C$ together with a homomorphism $\iota:G\to C$ with dense image.
Pilay conjectured that the Bohr compactification (i.e.~the universal compactification) of a pseudofinite group has abelian identity component (Conjecture 1.7 in \cite{pillay2015remarks}).
We answer this conjecture in the affirmative by the following result:

\begin{theorem}\label{thm:cpt_pseudo_fin_grp_ab}
	Let $G$ be a pseudofinite group. Then the identity component of any compactification $C$ of $G$ is abelian.
\end{theorem}

The proof is again just an easy application of Corollary \ref{cor:two_gen_com_fin_grp}.

\begin{proof}
	As $G$ is pseudofinite	it satisfies the statement of Corollary \ref{cor:two_gen_com_fin_grp} (and so does its image in $C$). An easy compactness argument shows that $C$ has the same property.
	Now let $\varrho_i:C\to\Aut(V_i)$ be the irreducible unitary representations of $C$ and $L_i$ the image of $\varrho_i$ ($i\in I$). By the Peter--Weyl Theorem, $C$ embeds continuously into $\prod_{i\in I}{L_i}$, and so $C^0$ embeds into $\prod_{i\in I}{L_i^0}$.
	
	But as $L_i$ is a compact quotient of $C$, Corollary \ref{cor:two_gen_com_fin_grp} holds in it, and so as in the proof of Theorem \ref{thm:Lie_grp_quot_prod_fin_grps} it follows that $L_i^0$ is abelian ($i\in I$). But then $C^0$ must be abelian as well, from the above embedding.
\end{proof}

\section*{Acknowledgements}

The second and third author want to thank Alessandro Carderi for interesting discussions. The content of this paper is part of the PhD project of the second author. This research was supported by ERC Consolidator Grant No.~681207.

After we finished a first version of this article and circulated it among some experts, it was pointed out that (independently and slightly earlier) Lev Glebsky found a solution to Zilber's problem along the same lines.

 
\end{document}